\documentclass{article}

\usepackage[preprint]{neurips_2022}

\usepackage[utf8]{inputenc} 
\usepackage[T1]{fontenc}    
\usepackage{hyperref}       
\usepackage{url}            
\usepackage{booktabs}       
\usepackage{amsfonts}       
\usepackage{nicefrac}       
\usepackage{microtype}      
\usepackage{xcolor}         
\usepackage{amsmath}
\usepackage{graphicx}
\usepackage{caption}
\usepackage{subcaption}
\usepackage{amsthm}
\usepackage{amssymb}
\hypersetup{colorlinks, citecolor=blue}

\usepackage{enumitem}

\usepackage{multirow, wrapfig}

\usepackage{lipsum}

\newcommand{\argmin}[1]{\underset{#1}{\mathrm{argmin}}}
\newcommand{\minimize}[1]{\underset{#1}{\mathrm{minimize}}}
\newcommand{\lmo}{\mathrm{LMO}}

\newcommand{\gap}{\mathbf{gap}}

\newcommand{\mD}{\mathcal D}

\newcommand{\mS}{\mathcal S}

\newcommand{\mN}{\mathcal N}
\newcommand{\R}{\mathbb R}

\newcommand{\bs}{\mathbf s}
\newcommand{\bx}{\mathbf x}

\newcommand{\bg}{\mathbf g}
\newcommand{\bz}{\mathbf z}
\newcommand{\bv}{\mathbf v}
\newcommand{\by}{\mathbf y}
\newcommand{\bh}{\mathbf h}

\newcommand{\FW}{\textsc{FW}}
\newcommand{\FWflow}{\textsc{FWFlow}}
\newcommand{\AvgFW}{\textsc{AvgFW}}
\newcommand{\AvgflowFW}{\textsc{AvgFWFlow}}

\newcommand{\sign}{\mathbf{sign}}

\newcommand{\conv}{\mathbf{conv}}

\newtheorem{theorem}{Theorem}[section]

\newtheorem{lemma}[theorem]{Lemma}

\newtheorem{corollary}[theorem]{Corollary}

\theoremstyle{definition}

\newcommand{\red}[1]{{\color{red}#1}}

\title{Accelerating Frank-Wolfe via Averaging Step Directions}

\author{%
  Zhaoyue Chen \\
  \texttt{zhaoychen@cs.stonybrook.edu} \\
  \AND
  Yifan Sun \\
  \texttt{ysun@cs.stonybrook.edu} \\
}

\begin{document}

\maketitle

\begin{abstract}
 The Frank-Wolfe method is a popular method in sparse constrained optimization, due to its fast per-iteration complexity. However, the tradeoff is that its worst case global convergence is comparatively slow, and importantly, is fundamentally slower than its flow rate--that is to say, the convergence rate is throttled by discretization error. 
  In this work, we consider a modified Frank-Wolfe where the step direction is a simple weighted average of past oracle calls. This method requires very little memory and computational overhead, and provably decays this discretization error term. 
 Numerically, we show that this method improves the convergence rate over several problems, especially after the sparse manifold has been detected. Theoretically, we show the method has an overall global convergence rate of $O(1/k^p)$, where $0< p < 1$; after manifold identification, this rate speeds to  $O(1/k^{3p/2})$. We also observe that the method achieves this accelerated rate from a very early stage, suggesting a promising mode of acceleration for this family of methods.
\end{abstract}
 
\section{Introduction}
The Frank-Wolfe (FW) method (or conditional gradient method) \citep{dunn1978conditional,frank1956algorithm} solves the constrained optimization problem 
\begin{equation}
\minimize{x\in \mD} \quad f(x)
\label{eq:main}
\end{equation}
via the repeated iterations
\[
    \begin{array}{rcl}
\bs_k &=& \argmin{s\in \mD} \; \nabla f(\bx_k)^T\bs, \\
\bx_{k+1} &=& \bx_k + \gamma_k (\bs_k - \bx_k).
    \end{array}
    \tag{\FW}
\]
Here, we assume that $\mD$ is a compact, convex set and $f:\R^n\to\R$ is an $L$-smooth $\mu$-strongly convex function. The operation in the first line is often referred to as the \emph{linear minimization oracle (LMO)}, and will be expressed succinctly as $\bs_k=:\lmo_\mD(\bx_k)$.
The FW method has been used in many machine learning applications, such as compressed sensing, recommender systems \citep{Freund2017AnEF}, image and video co-localization \citep{Joulin2014EfficientIA}, etc.

A major area of research concerns the slow convergence rate of the FW method. In particular, using a fixed step size sequence $\gamma_k = O(1/k)$, in general the FW method does not converge faster than $O(1/k)$ in objective value. It has been surmised that this is caused by \emph{zig-zagging}, e.g. periodic erratic behavior of the terms $\bs_k$, causing $\bx_k$ to have bad step directions even arbitrarily close to the solution. For example, this phenomenon was explored in \citet{chen2021continuous}, which showed that in its \emph{continuous} form, the method does not zigzag and also reaches an arbitrarily high convergence rate.
Moreover, even better multistep discretization methods cannot close this gap.
Overall, the conclusion was reached that the largest contributer to the slow convergence rate of FW is the discretization error, which may be reduced as a function of $\Delta$ (the discretization step length), but not  $k$ (the iteration counter).

\subsection{Assumptions and contributions} 

In this work, we investigate the discretization error term more carefully, and construct an \emph{averaged FW method} as a simple approach to reduce zig-zagging and decay this key discretization term. 

\paragraph{Assumptions.} Throughout this work, important assumptions are: compact and convex $\mD$ (so that the LMO is consistent) and $L$-smooth $f$ (so that gradients are computable and do not diverge wildly). The extra assumption of convexity is required for the global convergence rate, but is not needed for method implementation. Additionally, $\mu$-strong convexity is only needed for the accelerated local rate, and we believe is a proof artifact; in practice, we observe $\mu = 0$ achieves similar speedups.

\paragraph{Contributions.}
Overall, for strongly convex objectives, if averaging is done with coefficients $(\tfrac{c}{c+k})^p$ for constants $c \geq 3p/2+1$, $0< p < 1$, we show 
\begin{itemize}
    \item global convergence  of $O(1/k^p)$ for the method and $O(1/t^{1-p})$ for the flow, and
    \item local convergence  of $O(1/k^{3p/2})$ for the method and $O(\log(t)/t^{c})$ for the flow.
\end{itemize}

The differentiation between local and global convergence  is characterized by the identification of the sparse manifold ($\bar k$, where $\lmo(\bx_k) = \lmo(\bx^*)$ for all $k \geq \bar k$) \citep{hare2004identifying,sun2019we}; this is in similar spirit to works like \citet{liang2014local}, which characterize this behavior in proximal gradient methods and other sparsifying methods.
Overall, these results suggest improved behavior for sparse optimization applications, which is the primary beneficiary of this family of methods.

\subsection{Related works}

\paragraph{Accelerated FW.}
A notable work that highlights the notorious zig-zagging phenomenon is
\citet{lacoste2015global}, where an Away-FW method is proposed that cleverly removes offending atoms and improves the search direction. Using this technique, the method is shown to achieve linear convergence under strong convexity of the objective. The tradeoff, however, is that this method requires keeping  past atoms, which may incur an undesired memory cost.
A work that is particularly complementary to ours is \citet{garber2015faster}, which show an improved $O(1/k^2)$ rate when the \emph{constraint set} is strongly convex--this reduces zigzagging since solutions cannot lie in low-dimensional ``flat facets''.
Our work addresses the exact opposite regime, where we take advantage of ``flat facets'' in sparsifying sets (1-norm ball, simplex, etc). This allows the notion of \emph{manifold identification} as determining when suddenly the method behavior improves.

\paragraph{Averaged FW.} Several previous works have investigated \emph{gradient averaging} \citep{acceleratingFW,FWEquilibrium}. While performance seems promising, the rate was not improved past $O(\frac{1}{k})$. 
\citet{ding2020k} investigates \emph{oracle averaging} by solving small subproblems at each iteration to achieve optimal weights.
The work \citet{chen2021continuous} can also be viewed as an oracle averaging technique, where the weights are chosen inspired by higher order discretization methods.
In comparison, in this work the averaging is \emph{intentionally unintelligent}, and the goal is to see how much benefit can be gained simply through this basic approach.

\paragraph{Game theory and fictitious play.} A similar version of this method appeared in~\cite{hofbauer2009time}. Here, Frank-Wolfe methods are shown to be an instance of a Best Response method over linearized function information at each step. Furthermore, if we imagine players having partial information and instead drawing from a distribution of best responses, this becomes  fictitious play, which results in an averaged LMO method. Indeed, our averaged FW method can be viewed as a Best Response Fictitious Play over a uniform distribution of all past LMOs. Therefore, the convergence results in this paper can therefore be applied to the
 Best Response Fictitious Play  as well.
 
\paragraph{Sparse optimization}
Other works that investigate \emph{local convergence behavior} include  \citet{liang2014local,liang2017local,poon2018local}. Here, problems which have these two-stage regimes are described as having \emph{partial smoothness}, which allows for the low-dimensional solution manifold to have significance. 
Other works that investigate manifold identification include \citet{sun2019we,iutzeler2020nonsmoothness,nutini2019active, hare2004identifying}.
We investigate manifold identification and local convergence of Frank-Wolfe as follows. Since zig-zagging often appears when the solution is on a low-dimensional manifold, we differentiate a local convergence regime of when this manifold is ``identified", e.g. all future LMOs are drawn from vertices of this specific manifold. After this point, we show that convergence of both the proposed flow and method can be improved with a decaying discretization error, which in practice may be fast.

\section{The problem with vanilla Frank-Wolfe}
In order to motivate the new averaged method, let us quickly review the convergence proof for FW method and flow, below
\[
    \begin{array}{rcl}
s(t)& =& \lmo_\mD(-\nabla f(x(t)) \\[1ex]
\dot x(t) &=& \gamma(t) (s(t)-x(t))
    \end{array}
    \;\;(\FWflow),
    \qquad\qquad
     \begin{array}{rcl}
\bs_k& =& \lmo_\mD(-\nabla f(\bx_k) \\[1ex]
\bx_{k+1} &=& \bx_k + \gamma_k (\bs_k-\bx_k)
    \end{array}
    \;\;(\FW).
\]
Specifically, the \FWflow~\citep{jacimovic1999continuous} is the continuous-time trajectory for which the \FW~\citep{frank1956algorithm,dunn1978conditional} is its Euler's explicit discretization.
Note that flows are not an implementable method, but rather an analysis tool that bypasses discretization artifacts.

Consider now the trajectory of the objective flow loss $h(t):=f(x(t))-\min_{x\in\mD}f(x)$. Using the properties of the LMO and convexity of $f$, it can be shown that
\begin{equation}
\dot h(t) = \nabla f(x)^T\dot x(t) = \gamma(t)\nabla f(x)^T(s-x)
\leq-\gamma(t) h(t)
\label{eq:diffeq-vanilla-flow}
\end{equation}
and taking $\gamma(t) = \tfrac{c}{c+t}$, we arrive at 
\[
h(t) \leq \frac{h(0)}{(c+t)^c} = O(1/t^c).
\]
In contrast, using $L$-smoothness, the \FW~method 
satisfies the recursion
\begin{eqnarray*}
f(\bx_{k+1})-f(\bx_k) &\leq& \gamma_k\underbrace{\nabla f(\bx_k)^T(\bx_{k}-\bs_k)}_{\leq -\gap(x)} + \frac{L\gamma_k^2}{2}\underbrace{\|\bx_{k}-\bs_k\|_2^2}_{= (2D)^2}\\
\end{eqnarray*}
where $D = \max_{x\in \mD}\|x\|_2$ and the duality gap $\gap(\bx) \geq f(\bx)-f(\bx^*)$ for all convex $f$. Defining $\bh_k = \bh(\bx_k)$
\begin{equation}
\bh_{k+1}-\bh_k \leq -\gamma_k \bh_k + 2LD^2\gamma_k^2
\label{eq:diffeq-vanilla}
\end{equation}
which recursively gives $\bh_k = O(\tfrac{c}{c+k})$. 
Note the key difference in \eqref{eq:diffeq-vanilla} and its analogous terms in \eqref{eq:diffeq-vanilla-flow} is the extra $2LD^2\gamma_k^2 = O(1/k^2)$ term, which throttles the recursion from doing better than $\bh_k = O(1/k)$; in particular, $\bh_{k+1}-\bh_k \geq \Omega(1/k^2)$. 
One may ask if it is possible to bypass this problem by simply picking $\gamma_k = O(1/k^p)$ more aggressively, e.g. $p > 1$; however, then such a sequence becomes summable, and  convergence of $\bx_k\to \bx^*$ is not assured. Therefore, we must have a sequence $\gamma_k$ converging \emph{at least} as slowly as $O(1/k)$.

Thus, the primary culprit in this convergence rate tragedy is the bound $\|\bs_k-\bx_k\|_2=O(D)$ (nondecaying). This bound is also not in general loose; if $\bx^*$ is in the interior of $\mD$, then even at optimum, $\bs_k$ may be any vertex in $\mD$ and may bounce around the boundary of the set. 
As an example, consider the 1-D optimization $\displaystyle\min_{-1\leq x \leq 1} x^2$.
Then, assuming $\bx_0 \neq 0$, $\bs_k = \sign(-\bx_k)$ for all $k$, and
\[
\|\bx_k-\bs_k\|_2 = \|\bx_k + \sign(\bx_k)\|_2 \geq 1, \quad \forall k.
\]
That is to say, this error term \emph{does not decay in general, even when $\bx_k \approx \bx^*$}.

\newpage
{
\begin{wrapfigure}{r}{0.35\textwidth}
    \centering
\includegraphics[width=0.2\textwidth]{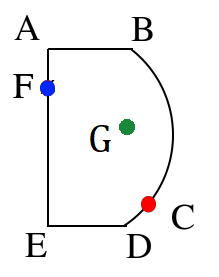}
    \caption{Example of a set $\mD$. Here, $A$, $B$, $D$, and $E$ are its isolated extremal vertices; and $F$ and $C$ are two candidate solution points.}
    \label{fig:vertices}
\end{wrapfigure}
Specifically, consider $\mS(\bx^*)$ the set of all possible LMOs at $\nabla f(\bx^*)$. \footnote{This set $\mS(\bx^*)$ is actually the subdifferential $\partial \sigma_\mD(-\nabla f(\bx^*))$, where $\sigma_\mD(z)$ is the support function of $\mD$ at $z$. Note that this is a well-defined quantity; although we do not require $f$ to be strongly convex (and thus $\bx^*$ may not be unique), we do require it to be $L$-smooth (and thus $\nabla f(\bx^*)$ \emph{is} unique).}
There are three possible scenarios.
\begin{itemize}[leftmargin=*]
    \item \textbf{Case I} (C): $\mS(\bx^*)$ contains only 1 element. This occurs when $\mD$ is a \emph{strongly convex  set}, like the 2-norm ball. It can also be a feature in parts of the boundary of other sets, like the 1,2-norm ball, group norm ball, etc.
    
    \item \textbf{Case II} (F): $\mS(\bx^*)$ contains multiple elements. This occurs when $\bx^*$ is on the boundary of $\mD$, which itself  is an \emph{atomic} set, e.g. the simplex, the $\ell_1$ ball, and linear transformations of these. This also includes  the nuclear norm ball and the group norm ball.
    
    \item \textbf{Case III} (G): $\bx^*$ is in the interior of any set, and $\mS(\bx^*)$ contains multiple elements, whether $\mD$ is strongly convex or not.
\end{itemize}
}
These three cases are illustrated in Figure \ref{fig:vertices}; at $C$, the LMO is necessarily unique, and spits out $\bs_k=\bx_k$; at $F$, the LMO may choose vertices $A$ or $E$, and will oscillate between them forever. Similarly, at $G$, the LMO will bounce between all points of which $G$ is its convex combination, forever.
\begin{figure}
\centering
\includegraphics[width=.8\textwidth]{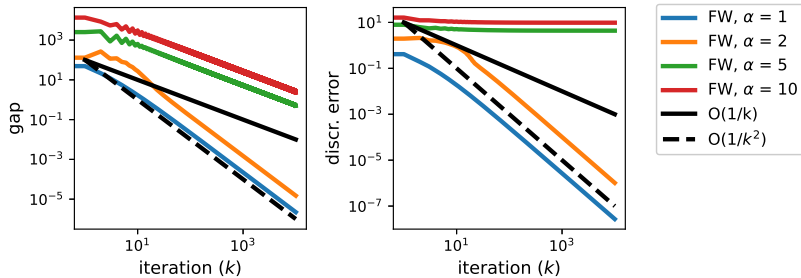}\\
\caption{An example of discretization error decay and method improvement when $\bx^*$ is on the boundary of a strongy convex set (in this case, the 2-norm ball). Unconstrained, $\|\bx^*\|_2 \approx 2.44$.  }
\label{fig:2normdecay}
\end{figure}
We quickly discuss Case I, addressed by past works \citep{garber2015faster}, where it is shown \FW~has a $O(1/k^2)$ (accelerated) convergence rate when $\bx^*$ is on the boundary of a strongly convex set.
The transition from case I to case III is shown in Figure \ref{fig:2normdecay}, where if $\alpha$ is small (and $\bx^*$ is on the boundary of the 2-norm ball) convergence is accelerated (and discretization error decays); on the other hand, when $\alpha$ is large (and $\bx^*$ is in the interior), we see no such acceleration.
Therefore, in cases II and III, an entirely new method must be defined to reduce this error term.

\section{An averaged Frank-Wolfe method}
We now propose an LMO-averaged Frank-Wolfe (AvgFW) method, by replacing $\bs_k$ with an averaged version $\bar \bs_k$. 
The proposed flow and method are as follows:

\[
\begin{array}{rcl}
s(t) &=& \lmo_\mD(x(t))\\[1ex]
{\dot {\bar s}(t)}&=&  \beta(t) (s(t)-\bar s(t))\\[1ex]
\dot x(t) &=&   \gamma(t) (\bar s(t)-x(t))
\end{array}
(\AvgflowFW),
\quad
\begin{array}{rcl}
\bs_k &=& \lmo_\mD(\bx_k)\\[1ex]
\bar \bs_{k} &=& \bar \bs_{k-1} + \beta_k (\bs_k-\bar \bs_{k-1})\\[1ex]
 \bx_{k+1} &=&  \bx_k + \gamma_k (\bar \bs_k - \bx_k)
\end{array}
(\AvgFW)
\]

where $\beta_k = (\frac{c}{c+k})^p$ for $c \geq 0$ and $0< p \leq 1$. \footnote{Recall that $\gamma_k = c/(c+k)$. While it is possible to use $\beta(t) = b/(b+t)$ with $b\neq c$ in practice, our proofs considerably simplify when the constants are the same.}
Here, the smoothing in $\bar s_k$ and $\bar \bs_k$ has two roles.
First, averaging reduces zigzagging; at every $k$, $\bar \bs_k$ is a convex combination
 of past $\bs_k$, and has a smoothing effect that qualitatively also reduces zig-zagging; this is of importance should the user wish to use line search or momentum-based acceleration.
 Second, this forces the discretization  term $\|\bx_{k}-\bar \bs_k\|_2$ to \emph{decay}, allowing for faster numerical performance.

 \begin{figure}
\centering
\includegraphics[width=.25\linewidth,trim={.5ex 0 .5ex 0},clip]{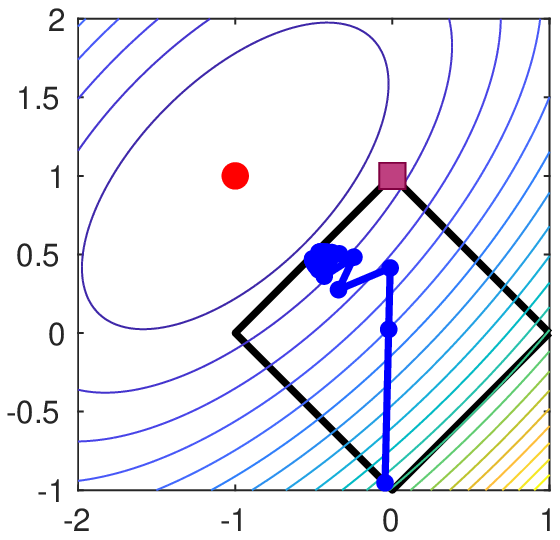}\hspace{1cm}
\includegraphics[width=.25\linewidth,trim={.5ex 0 .5ex 0},clip]{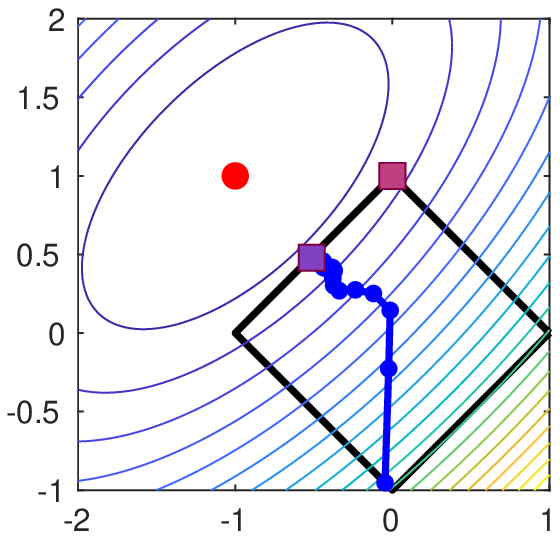}
\caption{\textit{Top}: \FW, with last  atom $\bs_k$ (light purple). \textit{Bottom}: \AvgFW, with last  atom $\bs_k$ (light purple), and  averaged  atom $\bar \bs_k$ (dark purple).}
\label{fig:trajectories}
\end{figure}
Figure \ref{fig:trajectories} shows a simple 2D example of our \AvgFW method (right) compared to the usual \FW method (left).
It is clear that  the averaged discretization term $\|\bar \bs_k-\bx_k\|_2$  decays, whereas the vanilla discretization term $\|\bs_k-\bx_k\|_2$ never converges.

\section{Convergence analysis}
\subsection{Global convergence}
We start by showing the method converges, with no assumptions on sparsity or manifold identification. While in practice we see faster convergence of \AvgFW compared to \FW, averaging produces challenges in the convergence proof.

Specifically, we require that the weights $\beta_k = (\frac{c}{c+k})^p$ be a little skewed toward more recent values; in practice, $p = 1$ works optimally, but in analysis we require $0 < p < 1$.
The convergence proof then expands a gap-like residual term as follows:
\[
\bar\bg_k := \nabla f(\bx_k)^T(\bar \bs_k-\bx_k) \leq- \beta_k  \gap(\bx_k) + (1-\beta_k)(1-\gamma_{k-1}) \bar \bg_{k-1} +  O(\gamma_k)
\]
and from there, form a recursion on the objective error. Unfortunately, that without relieving the $O(\gamma_k^2)$ constant term, this rate cannot possibly be better than that of the vanilla FW. This does contradict our numerical results, however, which show a global accelerated rate as well, suggesting that this rate could be tightened with the right assumptions. However, at this point our goal is to be fully general, so that the rates can then be used to bound manifold identification and discretization error decay.

\begin{theorem}[Global rates] Take $0\leq p \leq 1$.
\begin{itemize}\itemsep 0pt
    \item For $\beta(t) = \left(\frac{c}{c+t}\right)^p$, the flow \AvgflowFW~satisfies
$f(x(t)) \leq O\left(\frac{1}{t^{1-p}}\right).$
\item For $\beta_k = \left(\frac{c}{c+k}\right)^p$, the method \AvgFW~satisfies
$f(\bx_{k}) \leq O\left(\frac{1}{k^p}\right).$
\end{itemize}

\end{theorem}
The proofs are in Appendix \ref{app:sec:global}. Note that the rates are not exactly reciprocal. In terms of analysis, the method is that  allows the term $\beta_k\gap(\bx_k)$ to take the weight of an entire step, whereas in the flow, the corresponding term is infinitesimally small. This is a curious disadvantage of the flow analysis; since most of our progress is accumulated in the last step, the method exploits this more readily than the flow, where the step is infinitesimally small.

\subsection{Averaging decays discretization term}
Now that we have established the convergence rate of the method, we can now use it to bound the decay of the discretization term. 
There are two possible pitfalls here. First, note that, though $x(t)$ can be viewed as an  average of $s(t)$, the fact that $\|\dot x\|_2\to 0$ is not sufficient to suggest that $\bar s(t)$ ever stops moving (as is evidenced in the vanilla \FW~method, where $x(t)$ averages $s(t)$ which may oscillate forever). Conversely, the fact that $\bar s(t)$ is an average in itself is also not sufficient; the averaging rate of $O(1/k)$ is not summable, so though $\|\dot {\bar s}\|_2\to 0$, $\bar s$ may still never converge. (e.g., $\bar s(t) = \log(t)$.) 

In fact, we need both properties working in tandem. Intuitively, the fact that $x(t)$ indeed \emph{does} stop moving, halting at some $x^*$, means that if $\bar s(t)$ indeed keeps moving forever, it cannot do so in a divergent path, but must keep orbiting  $x^*$. Combined with the averaging will force $s(t)\to x^*$. \footnote{Intuitively, this is similar to the concept that   the alternating sequence $(-1)^k/k$ is summable but $1/k$ is not.} We prove this precisely in Appendix \ref{app:sec:averaging}; here, we provide a ``proof by picture'' 
(Figure \ref{fig:diamond_collapse}) when $\bs_k$ follows some prespecified trajectory (unrelated to any optimization problem). In one case, $\bx_k$ becomes stationary, and $\bar \bs_k\to \bx_k$; in the other case, $\bx_k$ keeps moving forever, suggesting  the averages never have the chance to catch up.   
    \begin{figure}
        \centering
\includegraphics[width=.475\textwidth,trim={1cm 0 1cm 0},clip]{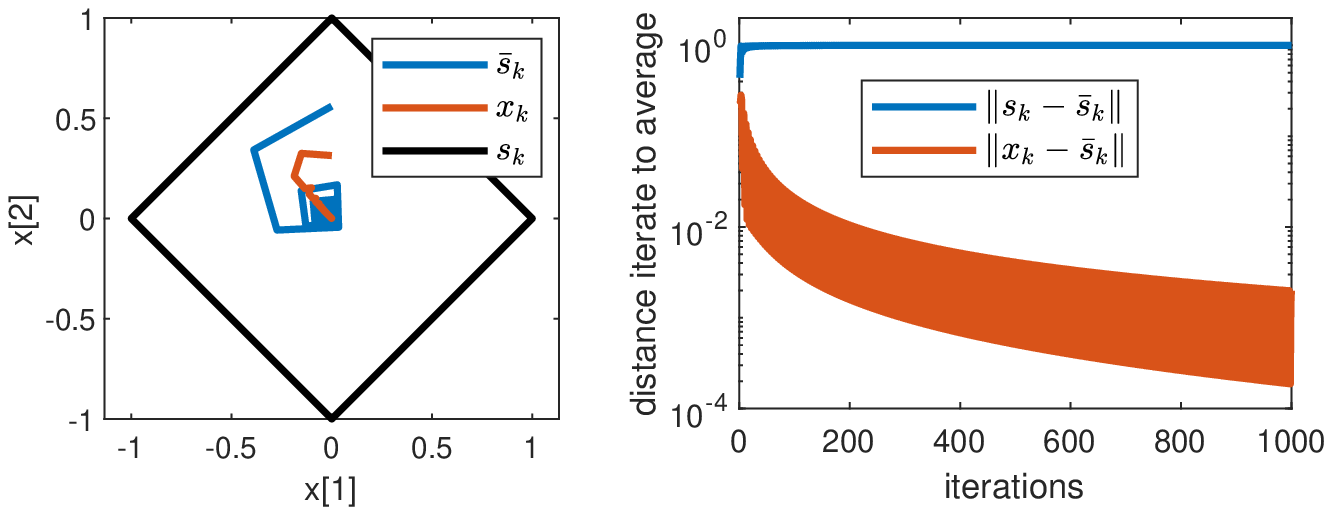}
\includegraphics[width=.475\textwidth,trim={1cm 0 1cm 0},clip]{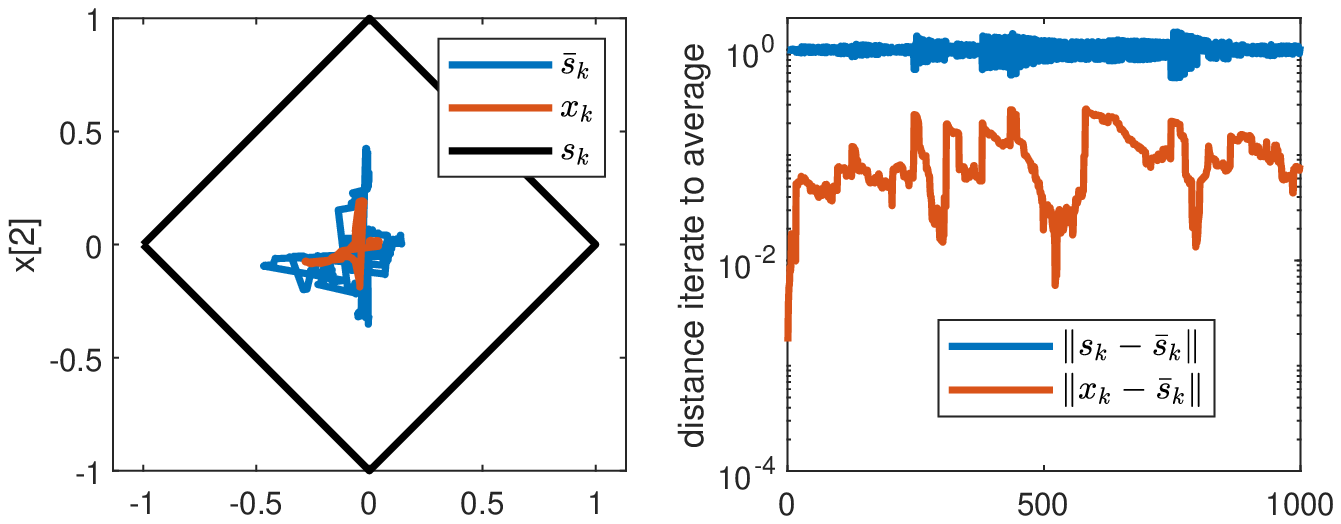}
        \caption{\textbf{Two examples over $\ell_1$ norm ball.}  \textit{Left:} When oscillation is a repeating sequence, the distance $\|\bx_k-\bar\bs_k\|_2$  decays at a $O(1/k)$ rate. \textit{Right:} When oscillation is random, this term does not decay in general.}
        \label{fig:diamond_collapse}
    \end{figure}

\subsection{Manifold identification}
Now that we know that the method converges and the discretization term collapses, we have an accelerated local convergence rate. Specifically, we consider  Case II: $\bx^*$ is on the boundary of $\mD$, on a low-dimensional facet, and is a combination of multiple (but not all) extremal vertices of $\mD$. Specifically, in this case, $\lmo_\mD(\bx^*)$ may return multiple values. This is the usual case in sparse optimization. In this case, we may  distinguish between two phases of the trajectory: first identifying the sparse manifold (global convergence), and then iterating after the sparse manifold has been identified (local convergence). 

A manifold identification \emph{rate} is expressed as $\bar k$, where for any $k > \bar k$, any $\lmo_\mD(\bx_k)$ is also an $\lmo_\mD(\bx^*)$. 
As an example, when $\mD$ is the one-norm ball, then at optimality, $\bx^*_i > 0$ only if $|\nabla f(\bx^*)_i| = \|\nabla f(\bx^*)\|_\infty$. We define a \emph{degeneracy parameter} $\delta$ as 
\[
\delta = \min_{j:\bx^*_j=0} \|\nabla f(\bx^*)\|_\infty - |\nabla f(\bx^*)_j|. 
\]
If $f(\bx_k)-f(\bx^*) \leq \delta/(Ln)$, then 
\[
\|\nabla f(x)-\nabla f(x^*)\|_\infty \leq n\|\nabla f(x)-\nabla f(x^*)\|_2\leq \frac{n}{L}(f(\bx_k)-f(\bx^*) ) \leq \delta/2.
\]
 Then, any LMO of $x$ must necessarily be a maximizer of $\bx^*$. Therefore, a manifold identification upper bound is $\bar k$ where for all $k > \bar k$, $f(\bx_k)-f(\bx^*) < \delta/(Ln)$; that is, the rate is the same as the global convergence rate, scaled by problem-dependent parameters.

This simple idea can be extended to any polyhedral set $\mD$, where the degeneracy parameter $\delta$ will depend on the $L$-smoothness of $f$ with respect to the gauge function of $\mD$ \citep{freund1987dual,friedlander2014gauge}. 
Concretely, we define $\mS(x)$ (the support of $x$ w.r.t. $\mD$) as the set of vertices in $\mD$ that may be retured by the LMO; that is,
\[
\mS(x) := \{s : s^T\nabla f(x) = \min_{s'\in \mD} {s'}^T\nabla f(x)\}
\]
and we say that the manifold has been identified at $\bar k$ if for all $k > \bar k$, $\lmo(\bs_k)\in \mS(\bx^*)$. A similar $\delta$ can be computed in this case. The point is, in all these cases, the manifold identification rate is simply a $1/\delta$ scaling of the method convergence rate. (See also \cite{sun2019we}.)
\footnote{Note that if $\mD$ is a strongly convex set, then if $\bx^*$ is on the boundary (Case I), then $\mS(\bx^*)$ is a singleton and moreover $\delta = 0$ and  manifold identification cannot happen in finite time. For this reason, this is not a case that benefits from this analysis; we really are only considering polyhedral $\mD$. Note that in Case III, manifold identification is simply identified from the very start.}

\subsection{Accelerated local convergence}

To see how averaging affects the convergence rate, first let us consider the flow of AvgFW:
\begin{eqnarray*}
\frac{\partial}{\partial t} f(x(t)) &=& \gamma(t) \nabla f(x(t))^T(\bar s(t)-x(t))\\
&=& \gamma(t) \underbrace{\nabla f(x(t))^T(\bar s(t)-\hat s(t))}_{A} + \gamma(t)\underbrace{\nabla f(x(t))^T(\hat s(t) - x(t))}_{B}
\end{eqnarray*}
where we carefully pick $\hat s(t)$ as follows:
\[
\dot{\hat s}(t) = \beta(t)(\tilde s(t) - \hat s(t)), \quad \tilde s(t) = \argmin{s\in \mS(x^*)}\|s(t)-s\|_2.
\]
In other words, once $\mS(x(t)) = \mS(x^*)$ (manifold identified), then $\tilde s(t) = s(t)$ for all subsequent $t$.
 Using this last choice of $\hat s(t)$, we are able to ensure $A$ decaying via averaging, and and $B = -\gap(x(t))$ after manifold identification.
\begin{theorem}[Local rate]
After manifold identification, the flow AvgFWFlow satisfies
\[
f(x(t))-f(x^*) \leq \frac{\log(t)}{t^c}.
\]
\end{theorem}
The proof is in appendix \ref{app:sec:local}.

Now let us extrapolate to the method. 
From $L$-smoothness of $f$, we have the difference inequality
\[
f(\bx_{k+1})-f(\bx_k) \leq \gamma_k\underbrace{\nabla f(\bx_k)^T(\bar \bs_k-\hat \bs_k)}_{A} +\gamma_k\underbrace{\nabla f(\bx_k)^T(\hat \bs_k-\bx_k)}_{B} + \frac{\gamma_k^2L}{2}\underbrace{\|\bar \bs_k-\bx_k\|_2^2}_{C}.
\]
The same tricks work on terms $A$ and $B$; however, as in the vanilla case, the rate is throttled by $f(\bx_{k+1})-f(\bx_k) \geq O(C)$. We now combine in our averaging result, which, mixed with the global convergence rate (which has no assumptions in discretization decay) shows that  $\|\bs_k-\bx_k\|^2_2 = O(1/k^{3p/2-1})$. This gives our final convergence result.

\begin{theorem}[Local rate]
Assume that $f$ is $\mu$-strongly convex, and pick $c \geq 3p/2+1$. After manifold identification, the method AvgFW satisfies
\[
f(\bx_k)-f(\bx^*) = O(1/k^{3p/2}).
\]
\end{theorem}
The proof is in appendix \ref{app:sec:local}.
Although the proof requires $p<1$, in practice, we often use $p = 1$ and observe about a $O(1/k^{3/2})$ rate, regardless of strong convexity.

\section{Numerical Experiments}

\subsection{Simulated Compressed Sensing}
We minimize a quadratic function with a $\ell_1$- norm ball constraint. Given $\bx_0\in \R^{m}$, a sparse ground truth vector with 10\% nonzeros, and given $A\in \R^{n\times m}$ with entries i.i.d. Gaussian, we generate $\by = A\bx_0  + \bz$ where $\bz_i\sim \mN(0,0.05)$. Then we solve 
\begin{equation}
\min_{x\in \R^n} \quad  \tfrac{1}{2}\|Ax-y\|_2^2\qquad
\mathrm{subject~to} \quad  \|x\|_1\leq\alpha.
\end{equation}
Figure \ref{fig:quadratic} evaluates the performance of \FW~and  \AvgFW, for a problem with with  $m = 100$, $n = 500$, and varying values of $\alpha$. Note that this problem is \emph{not} strongly convex. In all cases, we see an improvement in the duality gap (which upper bounds the objective suboptimality $f(\bx)-f(\bx^*)$) improve its convergence rate from $O(1/k)$ to approaching $O(1/k^{3/2})$, and in all cases the discretization rate of \AvgFW~is comparable to that of the gap. The support size itself varies based on $\alpha$, and when $\alpha$ is neither very big nor very small, seems to decay slowly; however, the final sparsity set and the initial working set (number of nonzeros touched from that point on) are very similar, suggesting that though the optimal manifold was not found exactly, it was well-approximated. 
\begin{figure}[ht!]
\centering
\includegraphics[width=\textwidth]{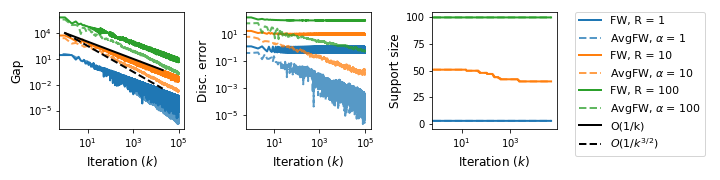}
        \caption{\textbf{Compressed sensing.}  \textit{Left:} Gap.  \textit{Center:} Discretization error,  $\|\bs_k-\bx_k\|_2$ for FW and $\|\bar\bs_k- \bx_k\|_2$ for AvgFW. \textit{Right:} Support size, which is the number of unique indices from current iteration till the end. }
        \label{fig:quadratic}
\end{figure}

\subsection{Sparse Logistic Regression}
We now evaluate our proposed \AvgFW~method on two binary classification tasks. We minimize
\begin{equation}
    \min_{x\in \R^n} \;  \frac{1}{m}\sum_{i=1}^{m}\log(1+\exp(-y_iz_i^Tx))\quad
    \mathrm{s.t.} \;  \|x\|_1\leq\alpha
\end{equation}
for two real world datasets \citep{guyon2004result}.
\begin{itemize}
\item \textbf{Low sparsity problem (Fig \ref{fig:gisette}).}
The Gisette task is to recognize  handwritten digits; to use binary classification, we only diffentiate between  4 and 9. The dataset is fully dense, with $n = 5000$ features, so the entry in Gisette dataset is dense. We use a 60/30 train/validation split over  $m=2000$ data samples, and we set $\alpha = 10$.

\item \textbf{High sparsity problem (Fig \ref{fig:dorothea}).}
We consider the task of predicting which chemical compounds bind to Thrombin from the Dorothea drug discovery dataset.  The dataset is sparse, with  about 0.91\% nonzeros, and the labels are unbalanced (9.75\% +1, 90.25\% -1). We use the given $m = 800$ training samples and $m = 350$ validation samples, with  $n=100000$ features, and $\alpha=10$. 

\end{itemize}
In both cases, we optimize $\alpha$ over the validation set, sweeping a coarse logarithmic grid of 10 points from 1 to 100. Here, we see similar speedups in the duality gap, discretization error, and support size.

\begin{figure}
\centering
\begin{subfigure}[b]{\textwidth}
\includegraphics[width=\textwidth,trim={3ex 3ex 3ex 2ex},clip]{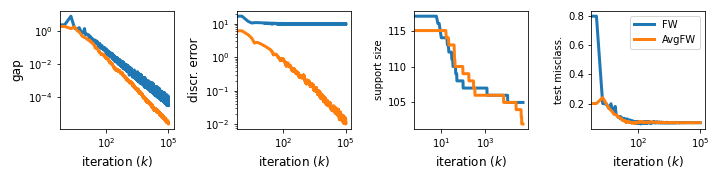}
        \caption{\textbf{Gisette.} }
        \label{fig:gisette}
\end{subfigure}
\begin{subfigure}[b]{\textwidth}
\centering
\includegraphics[width=\textwidth,trim={3ex 3ex 3ex 2ex},clip]{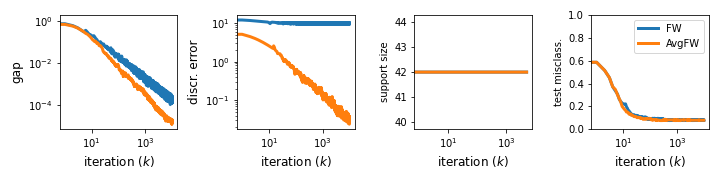}
        \caption{\textbf{Dorothea.} }
        \label{fig:dorothea}
\end{subfigure}
\caption{Performance of \AvgFW~on binary classification over real world datasets.}
\end{figure}


\section{Discussion}

The main goal of this study is to see if the discretization error, which seems to be the primary cause of slow \FW~convergence, can be attacked directly using an averaging method, and thereby speed up the convergence rate. This was accomplished; without using away steps we are able to improve the $O(1/k)$ rate to up to $O(1/k^{3/2})$ with negligible computation and memory overhead; moreover, the offending term in the discretization error, which is constant in \FW~when $\mD$ is not strongly convex, is shown here to always decay. A global and local convergence rate is given. 

Our numerical results show that, though our theoretical improvements are only local, the effect of this acceleration appears effective globally; moreover, manifold identification (or at least reduction to a small working set of nonzeros) appears almost immediately in many cases. In all cases, we do note that though averaging improves duality gap, it does not seem to appreciably improve manifold identification or test error performance, suggesting that a fast-and-loose implementation (no averaging) works well for all practical reasons. Still, we motivate that the improvement in the convergence rate is valuable in providing reliability guarantees in downstream applications.

\newpage
\bibliographystyle{plainnat}
\bibliography{main.bbl}
\newpage
\appendix

\section{Accumulation terms}
\label{app:sec:accumulation}
\begin{lemma}
For an averaging term $\bar s(t)$ satisfying
\[
\dot{\bar s}(t) =  \beta(t) (s(t) - \bar s(t)), \qquad \bar s(0) = s(0) = 0
\]
where $\beta(t) = \frac{c^p}{(c+t)^p}$, then
\[
\bar s(t) = 
\begin{cases}
\displaystyle e^{-\alpha(t)}\int_0^t \frac{c^p e^{\alpha(\tau)}}{(c+\tau)^p} s(\tau) d\tau, & p\neq 1\\
\displaystyle \frac{c}{(c + t)^c} \int_0^t  (c + \tau)^{c - 1} s(\tau) d\tau & p = 1
\end{cases}
\]
where $\alpha(t) = \frac{c^p(c+t)^{1-p}}{1-p}$. If $s(t) = 1$ for all $t$, then we have an accumulation term
\[
\bar s(t) = \begin{cases}
\displaystyle  1-\frac{e^{\alpha(0)}}{e^{\alpha(t)}},& p\neq 1\\
 \displaystyle 1-(\frac{c}{c+t})^c, & p = 1\\
 \end{cases}
 \]

\end{lemma}
\begin{proof}
This can be done through simple verification.
\begin{itemize}
    \item If $p\neq 1$,
\[
\alpha'(t) = \frac{c^p}{(c+t)^p} = \beta(t), 
\]
and via chain rule,
\[
\bar s'(t) = \underbrace{e^{-\alpha(t)}\frac{c^p\exp(\alpha(t))}{(c+t)^p}}_{=\beta(t)}s(t)-\alpha'(t) \underbrace{\exp(-\alpha(t))\int_0^t \frac{c^p\exp(\alpha(\tau))}{(c+\tau)^p} s(\tau) d\tau}_{\bar s(t)}.
\]
The accumulation term can be verified if 
\[
e^{-\alpha(t)}\int_0^t \frac{c^pe^{\alpha(\tau)}}{(c+\tau)^p} d\tau = 1-\frac{e^{\alpha(0)}}{e^{\alpha(t)}}
\]
which is true since 
\[
e^{-\alpha(t)}\int_0^t \frac{c^pe^{\alpha(\tau)}}{(c+\tau)^p} d\tau =  e^{-\alpha(t)}\int_0^t(\frac{d}{d\tau} e^{\alpha(\tau)})d\tau.
\]
\item If $p = 1$
\[
\bar s'(t) = \frac{c}{(c+t)}s(t) - \frac{c^2}{(c+t)^{c+1}}\int_0^t (c+\tau)^{c-1} s(\tau) d\tau = \frac{c}{(c+t)} (s(t)-\bar s(t)).
\]
For the accumulation term, 
\[
\frac{c}{(c+t)^c}\int_0^t (c+\tau)^{c-1}  d\tau = \frac{c}{(c+t)^c}\int_0^t (\frac{\partial}{\partial \tau} \frac{(c+\tau)^c}{c})  d\tau = 1-(\frac{c}{c+t})^c.
\]
\end{itemize}

\end{proof}
For convenience, we define
\[
\beta_{t,\tau} :=
\begin{cases}
\displaystyle   \frac{c^p e^{\alpha(\tau)-\alpha(t)}}{(c+\tau)^p}   , & p\neq 1\\
\displaystyle \frac{c(c+\tau)^{c-1}}{(c+t)^b}, & p = 1,
\end{cases}
\qquad
\bar\beta_t :=
\begin{cases}
\displaystyle  1-\frac{\exp(\alpha(0))}{\exp(\alpha(t))}, & p\neq 1\\
\displaystyle 1-(\frac{c}{c+t})^b & p = 1
\end{cases}
\]

\begin{lemma}
For the averaging sequence $\bar \bs_k$ defined recursively as 
\[
\bar \bs_{k+1} = \bar \bs_k + \beta_k (\bs_k - \bar \bs_k), \qquad \bar \bs_0 = 0.
\]
Then
\[
\bar \bs_k = \sum_{i=1}^k \beta_{k,i} \bs_i, \qquad
\beta_{k,i} =  \frac{c^p}{(c+i)^p}\prod_{j=0}^{k-i-1} \left(1-\frac{c^p}{(c+k-j)^p}\right) \overset{p=1}{=} \frac{c}{c+i} \prod_{j=0}^{c} \frac{i+j+1}{c+k-j}
\]
and moreover, $\sum_{i=1}^k \beta_{k,i} = 1$.
\end{lemma}

\begin{proof}

\begin{eqnarray*}
\bar\bs_{k+1} &=& \frac{c^p}{(c+k)^p} \bs_k + \left(1-\frac{c^p}{(c+k)^p}\right)\bar\bs_k\\
&=& \frac{c^p}{(c+k)^p} \bs_k + \frac{c^p}{(c+k-1)^p}\left(1-\frac{c^p}{(c+k)^p}\right) \bs_{k-1} + \left(1-\frac{c^p}{(c+k)^p}\right)\left(1-\frac{c^p}{(c+k-1)^p}\right)\bar\bs_{k-1}\\
&=& \sum_{i=0}^k \frac{c^p}{(c+k-i)^p}\prod_{j=0}^{i-1} \left(1-\frac{c^p}{(c+k-j)^p}\right)\bs_{k-i}\\
&\overset{l = k-i}{=}& \sum_{l=1}^k \underbrace{\frac{c^p}{(c+l)^p}\prod_{j=0}^{k-l-1} \left(1-\frac{c^p}{(c+k-j)^p}\right)}_{\beta_{k,l}}\bs_{l}.
\end{eqnarray*}
If $p = 1$, then 

\[
\beta_{k,i} = \frac{c}{c+i}\prod_{l=0}^{k-i-1} \frac{k-l}{c+k-l} =\frac{c}{c+i} \frac{k(k-1)(k-2)\cdots (i+1)}{(c+k)(c+k-1)\cdots (c+i+1)}
= \frac{c}{c+i} \prod_{j=0}^{c} \frac{i+j+1}{c+k-j} 
\]

For all $p$, to show the sum is 1, we do so recursively. At $k = 1$, $\beta_{1,1} = \frac{c^p}{(c+1)^p}$. Now, if $\sum_{i=0}^{k-1} \beta_{k-1,i} = 1$, then for $i \leq k-1$
\[
\beta_{k,i} =   \left(1-\frac{c^p}{(c+k)^p}\right)\beta_{k-1,i}, \quad i \leq k-1
\]
and for $i = k$, $\beta_{k,k} = \frac{c^p}{(c+k)^p}$. Then
\[
\sum_{i=1}^k \beta_{k,i} = \beta_{k,k} +  \left(1-\frac{c^p}{(c+k)^p}\right)\sum_{l=1}^{k-1} \beta_{k-1,i} = \frac{c^p}{(c+k)^p} + \left(1-\frac{c^p}{(c+k)^p}\right) = 1.
\]
\end{proof}

\section{Averaging}
\label{app:sec:averaging}

In the vanilla Frank-Wolfe method, we have two players $s$ and $x$, and as $x\to x^*$, $s$ may oscillate around the solution fascet however it would like, so that its average is $x^*$ but $\|s-x^*\|_2$ remains bounded away from 0. However, we now show that if we replace $s$ with $\bar s$, whose velocity slows down, then it must be that $\|s-x^*\|_2$ decays.

\begin{lemma}[Continuous averaging]
Consider some vector trajectory $v(t)\in \mathbb R^n$, and suppose 
\begin{itemize}
    \item $\|v(t)\|_2 \leq D$ for arbitrarily large $t$

\item  $\|v'(t)\|_2 = \beta(t) D$  

\item  $\frac{1}{2}\|\int_t^\infty \gamma(\tau) v(\tau) d\tau\|_2^2 = O(1/t^q)$  for $q > 0$.
\end{itemize}
Then $\|v(t)\|^2_2 \leq O(t^{q/2+p-1})$.

\end{lemma}

\begin{proof}

We start with the orbiting property.
\[\frac{d}{dt} \left(\frac{1}{2}\|\int_t^\infty \gamma(\tau) v(\tau) d\tau\|_2^2\right) = -\int_t^\infty \gamma(t)\gamma(\tau) v(\tau)^Tv(t) d\tau \leq 0.
\]
Since this is happening asymptotically, then the negative derivative of the LHS must be upper bounded by the negative derivative of the RHS. That is, if a function is decreasing asymptotically at a certain rate, then its negative derivative should be decaying asymptotically at the negative derivative of this rate. So,

\[
\int_t^\infty \gamma(\tau) v(\tau)^Tv(t) d\tau \leq  O(1/t^q).
\]
This indicates either that $\|v(t)\|_2$ is getting smaller (converging) or $v(t)$ and its average are becoming more and more uncorrelated (orbiting).

Doing the same trick again with the negative derivative,
\[
-\frac{d}{dt}\int_t^\infty \gamma(\tau) v(\tau)^Tv(t) d\tau  = \gamma(t)\|v(t)\|_2^2 - \int_t^\infty \gamma(\tau) v(\tau)^Tv'(t)  d\tau
\]
By similar logic, this guy should also be decaying at a rate $O(1/t^{q+1})$, so 

\[
\gamma(t)\|v(t)\|_2^2  \leq \frac{C_2}{t^{q+1}} +  \int_t^\infty \gamma(\tau) v(\tau)^Tv'(t)  d\tau
\leq \frac{C_2}{t^{q+1}} +  \underbrace{\|\int_t^\infty \gamma(\tau) v(\tau)d\tau\|_2}_{\leq O(1/t^{q/2})}D\beta(t) = \frac{C_2}{t^{q+1}} + \frac{C_3}{t^{q/2+p}}
\]

Therefore 
\[
\|v(t)\|_2^2  \leq  \frac{C_2}{t^{q}} + \frac{C_3}{t^{q/2+p-1}} = O(\frac{1}{t^{q/2+p-1}}).
\]
\end{proof}

\begin{corollary}
Suppose $f$ is $\mu$-strongly convex. Then 
\[
\|\bar s(t)-x(t)\|_2^2 \leq Ct^{-(q/2+p-1)}
\]
for some constant $C>0$.
\end{corollary}
\begin{proof}
Taking $v(t) = \bar s(t)-x(t)$, it is clear that if $\beta(t) \geq \gamma(t)$ then the first two conditions are satisfied. 
In the third condition, note that 
\[
\int_t^\infty \gamma(\tau) (\bar s(\tau)-x(\tau))d\tau = \int_t^\infty \dot x(\tau) d\tau = x^*-x(t)
\]
and therefore
\[
\frac{1}{2}\|\int_t^\infty \gamma(\tau) v(\tau) d\tau\|_2^2 = \frac{1}{2}\|x^*-x(t)\|_2^2 \leq \mu (f(x)-f^*)
\]
by strong convexity.
\end{proof}

\begin{lemma}[Discrete averaging]
Consider some vector trajectory $\bv_k\in \mathbb R^n$. Then the following properties cannot all be true.
\begin{itemize}
    \item $\|\bv_k\|_2 \leq D$ for arbitrarily large $k$

\item  $\|\bv_{k+1}-\bv_k\|_2 \leq \beta_k D$  

\item  $\frac{1}{2}\|\sum_{i=k}^\infty  \gamma_k \bv_k\|_2^2 \leq \frac{C_1}{k^q}$  for $q > 0$.
\end{itemize}
Then $\|\bv_k\|^2_2 \leq O(k^{q/2+p-1})$.

\end{lemma}

\begin{proof}
The idea is to recreate the same proof steps as in the previous lemma. Note that the claimis not that these inequalities happen at each step, but that they must hold asymptotically in order for the asymptotic decay rates to hold. So
\[ 
\frac{1}{2}\|\sum_{i=k}^\infty \gamma_i \bv_i\|_2^2  -\frac{1}{2}\|\sum_{i=k}^\infty \gamma_{i+1} \bv_{i+1}\|_2^2 
=\frac{\gamma_k^2}{2}\|\bv_k\|_2^2 + \gamma_k\bv_k^T\left(\sum_{i=k}^\infty \gamma_{i+1} \bv_{i+1}\right)  \leq \frac{C_1}{(k+1)^q}-\frac{C_1}{k^q}  =\frac{C_2}{k^{q}}
\]

and therefore
\[
\frac{\gamma_k}{2}\|\bv_k\|_2^2 + \bv_k^T\left(\sum_{i=k}^\infty \gamma_{i+1} \bv_{i+1}\right) \leq \frac{C_1}{(k+1)^{q+1}}-\frac{C_1}{k^{q+1}} = \frac{C_2}{k^{q+1}}.
\]
Next,
\begin{eqnarray*}
&&\bv_k^T\left(\sum_{i=k}^\infty \gamma_{i+1} \bv_{i+1}\right)  -  \bv_{k+1}^T\left(\sum_{i=k+1}^\infty \gamma_{i+1} \bv_{i+1}\right) + \bv_{k}^T\left(\sum_{i=k+1}^\infty \gamma_{i+1} \bv_{i+1}\right)- \bv_{k}^T\left(\sum_{i=k+1}^\infty \gamma_{i+1} \bv_{i+1}\right) \\
&=&\gamma_{k+1} \bv_k^T  \bv_{k+1}  + (\bv_k-\bv_{k+1})^T\left(\sum_{i=k+1}^\infty \gamma_{i+1} \bv_{i+1}\right)
\end{eqnarray*}

\begin{eqnarray*}
\frac{\gamma_k}{2}\|\bv_k\|_2^2 - 
\frac{\gamma_{k+1}}{2}\|\bv_{k+1}\|_2^2 +\bv_k^T\left(\sum_{i=k}^\infty \gamma_{i+1} \bv_{i+1}\right)  -  \bv_{k+1}^T\left(\sum_{i=k+1}^\infty \gamma_{i+1} \bv_{i+1}\right) =\\
\frac{\gamma_k}{2}\|\bv_k\|_2^2 \underbrace{- 
\frac{\gamma_{k+1}}{2}\|\bv_{k+1}\|_2^2
+
\gamma_{k+1} \bv_k^T  \bv_{k+1}}_{-\frac{\gamma_{k+1}}{2}\|\bv_{k+1}-\bv_k\|_2^2 + \frac{\gamma_{k+1}}{2}\|\bv_k\|_2^2}\  +  (\bv_k-\bv_{k+1})^T\left(\sum_{i=k+1}^\infty \gamma_{i+1} \bv_{i+1}\right) \leq \frac{C_3}{k^{q+1}}
\end{eqnarray*}
Therefore 
\[
\frac{\gamma_k+\gamma_{k+1}}{2}\|\bv_k\|_2^2  \leq \frac{C_3}{k^{q+1}} + \underbrace{(\bv_{k+1}-\bv_k)^T\left(\sum_{i=k+1}^\infty \gamma_{i+1} \bv_{i+1}\right)}_{O(\beta_k/k^{q/2})} + \underbrace{\frac{\gamma_{k+1}}{2}\|\bv_{k+1}-\bv_k\|_2^2}_{O(\gamma_k\beta_k^2)}
\]
Finally,
\[
\|\bv_k\|_2^2  \leq  \frac{C_3}{k^{q}} + \frac{C_4}{k^{q/2+p-1}} + \frac{C_5}{k^{2p}} = O(1/k^{\min\{q/2+p-1,2p\}}).
\]
\end{proof}

\begin{corollary}
Suppose $f$ is $\mu$-strongly convex. Then if $f(x)-f^* = O(k^{-q})$
\[
\|\bar \bs_k-\bx_k\|_2^2 \leq C\max\{k^{-(q/2+p-1)},k^{-2p}\} 
\]
for some constant $C>0$.
\end{corollary}
\begin{proof}
Taking $\bv_k = \bar \bs_k-\bx_k$, it is clear that if $\beta(t) \geq \gamma(t)$ then the first two conditions are satisfied. 
In the third condition, note that 
\[
\sum_{i=k}^\infty \gamma_i (\bar \bs_i-\bx_i) = \sum_{i=k}^\infty \bx_{i+1}-\bx_i = \bx^*-\bx_k
\]
and therefore
\[
\frac{1}{2}\|\sum_{i=k}^\infty \gamma_i (\bar \bs_i-\bx_i)\|_2^2 = \frac{1}{2}\|\bx^*-\bx_k\|_2^2 \leq \mu (f(\bx_k)-f^*)
\]
by strong convexity.
\end{proof}

\section{Global rates}
\label{app:sec:global}

\begin{lemma}[Continuous energy function decay]
\label{lem:continuous_energy}
Suppose $c \geq q-1$, and
\[
 g(t) \leq   -\int_0^t \frac{\exp(\alpha(\tau))}{\exp(\alpha(t))}\frac{(c+\tau)^c}{(c + t)^{c}} \frac{C_1}{(b+\tau)^r}d\tau
\]
Then
\[
g(t) \leq -\frac{\left(1-\frac{\alpha(1)}{\exp(\alpha(1))}\right)C_1}{\alpha(t)(1-p)}(c+t)^{1-r} 
\]
 
\end{lemma}
\begin{proof}

\begin{eqnarray*}
\frac{g(t)}{C_1}\exp(\alpha(t))(c+t)^{c}&\leq&- \int_0^t \exp(\alpha(\tau)) (c+\tau)^{c-r} d\tau\\
&=& - \int_0^t \sum_{k=1}^\infty \frac{\alpha(\tau)^k}{k!}(c+\tau)^{c-r} d\tau\\
&=& -\int_0^t \sum_{k=1}^\infty \frac{c^{kp}}{k!(1-p)^k}(c+\tau)^{k-pk+c-r} d\tau\\
&\overset{\text{Fubini}}{=}&  -\sum_{k=1}^\infty \frac{c^{kp}}{(1-p)^kk!}\int_0^t (c+\tau)^{k-pk+c-r} d\tau\\
&=&  - \sum_{k=1}^\infty \frac{c^{kp}}{(1-p)^kk!} \frac{(c+t)^{k-pk+c-r+1}-c^{k-pk+c-r+1}}{k-pk+c-r+1} \\
&=&   -\sum_{k=1}^\infty \frac{1}{(k+1)!}\left(\frac{c^{p}}{(c+t)^{p-1}(1-p)}\right)^k(c+t)^{1+c-r} \underbrace{\frac{1}{(1-p)+(c-r+1)/k}\frac{k+1}{k}}_{\geq C_2} \\
&\leq& -C_2 \sum_{k=1}^\infty \frac{(c+t)^{1+c-r}}{(k+1)!}\frac{\alpha(t)^{k+1}}{\alpha(t)}\\
&=& -\frac{C_2(c+t)^{1+c-r} }{\alpha(t)}(\exp(\alpha(t))-\alpha(1))
\end{eqnarray*}
Then

\begin{eqnarray*}
g(t) &\leq& -\frac{C_1C_2}{\alpha(t)}(c+t)^{1-r}\left(1-\frac{\alpha(1)}{\exp(\alpha(t))}\right)\\
 &\leq& -\frac{C_1C_2}{\alpha(t)}(c+t)^{1-r}\left(1-\frac{\alpha(1)}{\exp(\alpha(1))}\right)\\
 &\leq& -\frac{C_1C_3}{\alpha(t)}(c+t)^{1-r}
\end{eqnarray*}
where $C_3 = C_2\left(1-\frac{\alpha(1)}{\exp(\alpha(1))}\right)$  and $C_2 = \frac{1}{1-p}$ satisfies the condition.
\end{proof}
\begin{theorem}[Continuous global rate]
Suppose $0 < p < 1$. 
Then the averaged FW flow decays as $O(1/t^{1-p})$.
\end{theorem}

\begin{proof}
Consider the error function 
\[
g(t) := \nabla f(x(t))^T(\bar s(t)-x(t)), \qquad g(0) = 0.
\]
\begin{eqnarray*}
\dot g(t) &=& \frac{\partial}{\partial t} \nabla f(x)^T(\bar s-x)\\
&=& \left(\frac{\partial}{\partial t}\nabla f(x)^T \right)(\bar s - x)
+
\nabla f(x)^T\left(\frac{\partial}{\partial t}(\bar s - x)\right)
 \\
&=&  \underbrace{ \left(\frac{\partial}{\partial t} \nabla f(x)^T \right)}_{={\dot x}^T\nabla^2 f(x)}(\bar s - x)
+
\nabla f(x)^T \left(\beta(t) (s(t)-\bar s(t)) - \gamma(t)(\bar s(t)-x(t))\right)\\
&\leq & \gamma \underbrace{(\bar s - x)^T\nabla^2 f(x) (\bar s - x)}_{\leq 4LD^2\gamma(t)} + \underbrace{\beta(t)\nabla f(x)^T(s(t)-x(t))}_{-\beta(t)\gap(x)} - (\beta(t)+\gamma(t)) \underbrace{\nabla f(x)^T (\bar s(t) -x(t))}_{=g(t)} 
\end{eqnarray*}

\begin{eqnarray*}
g(t) &\leq&    \int_0^t \frac{\exp(\alpha(\tau))}{\exp(\alpha(t))}\frac{(c+\tau)^c}{(c + t)^{c}} \underbrace{\left(\frac{4LD^2c}{c+\tau} -  \frac{b^p}{(b+\tau)^p}\gap(x(\tau))\right)}_{A(\tau)}d\tau\\
\dot h(t) &\leq&  \gamma(t)g(t) \leq  \gamma(t) \int_0^t \underbrace{\frac{\exp(\alpha(\tau))}{\exp(\alpha(t))}\frac{(c+\tau)^c}{(c + t)^{c}}}_{\mu(\tau)} A(\tau) d\tau 
\end{eqnarray*}
In order for $h(t)$ to decrease, it must be that $\dot h(t) \leq 0$. However, since $\mu(\tau) \geq 0$ for all $\tau \geq 0$, it must be that $A(\tau) \leq 0$, e.g.
\[
  \frac{c^p}{(c+\tau)^p}\gap(x(\tau))\geq \frac{4LD^2c}{c+\tau}.
\]
which would imply $h(t) = O(1/(c+t)^{1-p})$. Let us therefore test the candidate solution
\[
h(t) = \frac{C_3}{(c+t)^{1-p}}.
\]
Additionally, from Lemma \ref{lem:continuous_energy}, if 
\[
A(\tau) \leq -\frac{C_1}{(c+t)}
\quad \Rightarrow\quad
g(t) \leq -\frac{C_1}{\alpha(t)(1-p)}
\]
and therefore 
\begin{eqnarray*}
\dot h(t) &\leq&  \gamma(t)g(t)  \leq -\frac{c}{c+t} \frac{C_1}{c^{p}}\cdot(c+t)^{p-1} \\
\int_0^t \dot(h(\tau))d\tau&\leq& \frac{C_1 c}{(1-p)c^p}(c+t)^{p-1}
\end{eqnarray*}
which satisfies our candidate solution for $C_3 = \frac{C_1 c}{(p-1)c^p}$.
\end{proof}

This term $ (\bar s - x)^T\nabla^2 f(x) (\bar s - x)\leq 4LD^2\gamma(t)$ is an important one to consider when talking about local vs global distance. The largest values of the Hessian will probably not correspond to the indices that are ``active'', and thus this bound is very loose near optimality.

\begin{lemma}[Discrete energy decay]
\label{lem:discrete_energy}
Suppose $0 < p < 1$. 
Consider the error function 
\[
\bg_k := \nabla f(\bx_k)^T(\bar\bs_k-\bx_t).
\]

Then
\[
\bg_k \leq -\sum_{i=0}^{k-1} \beta_{i,i} (\frac{i+1+c}{k+c})^c \gap(\bx_i)   - \beta_{k,k}\gap(\bx_k)+  \frac{4D^2L C_1}{(k+c)^p} + (\frac{c}{k+c})^c \bg_0.
\]
where $C_1 = c^p(1+\frac{1}{(c+1)(1-p)-1})$.
\end{lemma}
Importantly, $C_1$ is finite only if $p < 1$. When $p = 1$, the right hand side is at best bounded by a constant, and does not decay, which makes it impossible to show method convergence.

\begin{proof}
Define $\bz_k = \nabla f(\bx_k)$, $\bg_k = \bz_{k}^T(\bar\bs_{k}-\bx_{k})$.  Then 

\begin{eqnarray*}
\bg_k &=&\underbrace{\beta_{k,k} \bz_k^T(\bs_k-\bx_k) }_{-\beta_{k,k}\gap(\bx_k)} 
+ \underbrace{\sum_{i=0}^{k-1}\beta_{k,i}\bz_k^T(\bs_i-\bx_k)}_{A}\\
A&=& \sum_{i=0}^{k-1}\underbrace{\frac{\beta_{k,i}}{\beta_{k-1,i}}}_{=(1-\frac{c}{(c+i)^p})}\beta_{k-1,i} \bz_k^T(\bs_i-\bx_k)\leq  (1-\frac{c^p}{(c+k)^p}) \underbrace{\bz_k^T(\bar\bs_{k-1}-\bx_k)}_{=B}  \\
B&=&  \bz_k^T(\bar\bs_{k-1}-\underbrace{(\bx_{k-1}+\gamma_{k-1}(\bar\bs_{k-1}-\bx_{k-1}))}_{\bx_{k}}) \\
&=& (1-\gamma_{k-1})\bz_{k}^T(\bar\bs_{k-1}-\bx_{k-1})\\
&=& (1-\gamma_{k-1})(\bz_{k}-\bz_{k-1})^T\underbrace{(\bar\bs_{k-1}-\bx_{k-1})}_{(\bx_k-\bx_{k-1})\gamma_{k-1}^{-1}}
+ (1-\gamma_{k-1})\underbrace{\bz_{k-1}^T(\bar\bs_{k-1}-\bx_{k-1})}_{\bg_{k-1}} \\
&=& \frac{(1-\gamma_{k-1})}{\gamma_{k-1}} \underbrace{(\bz_{k}-\bz_{k-1})^T (\bx_k-\bx_{k-1})}_{\leq L\|\bx_k-\bx_{k-1}\|_2^2}
+ (1-\gamma_{k-1}) \bg_{k-1} \\
&\leq& (1-\gamma_{k-1})\gamma_{k-1}L \underbrace{\|\bar \bs_{k-1}-\bx_{k-1}\|_2^2}_{4D^2}
+ (1-\gamma_{k-1}) \bg_{k-1} \\
\end{eqnarray*}
Overall,
\begin{eqnarray*}
\bg_k &\leq& -\beta_{k,k}\gap(\bx_k) + 4D^2L\gamma_{k-1}(1-\gamma_{k-1})(1-\beta_k)+\underbrace{(1-\beta_k)(1-\gamma_{k-1})}_{\mu_k}\bg_{k-1}\\
&=& -\beta_{k,k}\gap(\bx_k) + 4D^2L \gamma_{k-1}\mu_k + \mu_k\bg_{k-1}\\
&=& -\beta_{k,k}\gap(\bx_k) + 4D^2L  \gamma_{k-1} \mu_k
-\beta_{k-1,k-1}\mu_k\gap(\bx_{k-1}) + 4D^2L \gamma_{k-2}\mu_{k-1}\mu_k + \mu_k\mu_{k-1}\bg_{k-2} 
\\
&=& -\sum_{i=0}^{k-1}\beta_{i,i}\gap(\bx_i)\prod_{j=i+1}^{k}\mu_j - \beta_{k,k}\gap(\bx_k) + 4D^2L\sum_{i=0}^k\gamma_{k-i}\prod_{j=i}^k\mu_j+ \prod_{j=1}^k\mu_k \bg_0
\end{eqnarray*}

Now we compute $\prod_{j=i}^k \mu_j$
\[
\prod_{j=i}^k(1-\gamma_{j-1}) = \prod_{j=i}^k\frac{j-1}{c+j-1} = \prod_{j=0}^c \frac{i-1+j}{k+j} \leq  (\frac{i+c}{k+c})^c
\]
Using $1-\frac{c^p}{(c+k)^p}\leq \exp(-(\frac{c}{c+k})^p)$, 
\begin{eqnarray*}
\log(\prod_{j=i}^k(1-\frac{c^p}{(c+j)^p})) \leq -\sum_{j=i}^k (\frac{c}{c+j})^p \leq  -\int_{i}^k (\frac{c}{c+j})^p dj 
=  \frac{c^p}{p+1} ((c+i)^{p+1}-(c+k)^{p+1})
\end{eqnarray*}
and therefore
\[
\prod_{j=i}^k(1-\frac{c^p}{(c+k)^p}) \leq  \frac{\exp(\frac{c^p}{p+1} (c+i)^{p+1})}{\exp(\frac{c^p}{p+1}(c+k)^{p+1})}
\]
which means
\[
\prod_{j=i}^k\mu_j \leq (\frac{i+c}{k+c})^c \exp(\frac{c^p}{p+1} ((c+i)^{p+1}-(c+k)^{p+1})).
\]
Now we bound the constant term coefficient.
\begin{eqnarray*}
\sum_{i=0}^k \gamma_{i-1}\prod_{j=i}^k\mu_j &\leq&
\sum_{i=0}^k\underbrace{\frac{c}{(c+i-1)}  (\frac{i+c}{k+c})^c}_{\text{max at $i=0$}} \underbrace{\exp(\frac{c^p}{p+1} ((c+i)^{p+1}-(c+k)^{p+1}))}_{\leq \frac{1}{\left(k-i\right)^{\left(c+1\right)\left(1-p\right)}}}\\
&\overset{C-S}{\leq}&\frac{c}{c-1}\frac{c^c}{(k+c)^c}  \sum_{i=0}^{k-1}  \frac{1}{\left(k-i\right)^{(c+1)(1-p)}}\\
&\leq&\frac{c}{c-1}\frac{c^c}{(k+c)^c}  \frac{1}{(c+1)(1-p)-1} \frac{1}{(1 - k)^{ (c+1)(1-p)-1}}\\
&\leq& \frac{C_1}{(k+c)^p}\frac{1}{(1 - k)^{ (c+1)(1-p)-1}}
\end{eqnarray*}
where $(*)$ if $c$ is chosen such that $(c+1)(1-p) > 1$ and $C_1>0$ big enough. Note that necessarily, $p < 1$, and the size of $C_1$ depends on how close $p$ is to 1.

Also, to simplify terms,
\[
\prod_{j=i}^k\mu_j \leq (\frac{i+c}{k+c})^c \underbrace{\exp(\frac{c^p}{p+1} ((c+i)^{p+1}-(c+k)^{p+1})).}_{\leq 1}
\]

Now, we can say
\[
\bg_k \leq -\sum_{i=0}^{k-1} \beta_{i,i} (\frac{i+1+c}{k+c})^c \gap(\bx_i)   - \beta_{k,k}\gap(\bx_k)+  \frac{4D^2L C_1}{(k+c)^p} + (\frac{c}{k+c})^c \bg_0.
\]

\end{proof}

\begin{theorem}[Global rate, $p < 1$.]
\label{th:globalrate}
Suppose $0 < p < 1$ and 
$c \geq \frac{c-1}{c^p}$. Then
 $h(\bx_k) =: \bh_k = O(\tfrac{1}{(k+c)^p})$.
\end{theorem}

\begin{proof}Start with
\[
\bg_k \leq -\sum_{i=0}^{k-1} \beta_{i,i} (\frac{i+1+c}{k+c})^c \gap(\bx_i)   - \beta_{k,k}\gap(\bx_k)+  \frac{4D^2L C_1}{(k+c)^p} + (\frac{c}{k+c})^c \bg_0.
\]
\begin{eqnarray*}
\bh_{k+1}-\bh_k &\leq& \gamma_k\bg_k + 2\gamma_k LD^2\\
&\leq & -\frac{c}{c+k}\frac{1}{(k+c)^c}\sum_{i=0}^{k-1} \beta_{i,i} (i+c+1)^c  \bh_i - \frac{c}{c+k}\beta_{k,k}\bh_k\\
 &&\qquad + 2D^2L\gamma_k( \frac{2C_1}{(k+c)^p} + \gamma_k) + \frac{c}{c+k}\frac{c^c}{(k+c)^c} \bg_0 \\
 &\leq&  -\frac{c^{p+1}}{c+k}\frac{1}{(k+c)^c}\sum_{i=0}^{k-1} (i+c)^{c-p}  \bh_i - \frac{c}{c+k}\beta_{k,k}\bh_k\\
 &&\qquad + 2D^2L\gamma_k( \frac{2C_1}{(k+c)^p} + \gamma_k) + \frac{c}{c+k}\frac{c^c}{(k+c)^c} \bg_0 \\
\end{eqnarray*}

Suppose $\bh_{k} \leq \frac{C_2}{(k+c)^p}$. Then
\begin{eqnarray*}
\bh_{k+1}  -\bh_k &\leq&   \underbrace{-\frac{c^{p+1}}{c+k}\frac{C_2}{(k+c)^c}\sum_{i=0}^{k-1} (c+i)^{c-2p}    - \frac{cC_2}{c+k}\frac{c^p}{(c+k)^{2p}}}_{-AC_2}\\
 &&\qquad + \underbrace{2D^2L\frac{c}{c+k}\left( \frac{2C_1}{(k+c)^p} + \frac{c}{c+k}\right) + \frac{c}{c+k}\frac{c^c}{(k+c)^c} \bg_0}_{B} \\
 B&=& \frac{c}{c+k}\left(2D^2L\left( \frac{2C_1}{(k+c)^p} + \frac{c}{c+k}\right) + \frac{c^c}{(k+c)^c} \bg_0\right)
 \\
 &\leq &\frac{2cD^2L( 2C_1+c ) + c^{c+1} \bg_0}{(c+k)^{p}(c+k)}\\
 &=:&\frac{C_3}{(c+k)^{p+1}}\\
\end{eqnarray*}
where $c > 1$.  Then,

\begin{eqnarray*}
 (k+c)^cA &=&  \frac{c^{p+1}}{c+k}\sum_{i=0}^{k-1} (c+i)^{c-2p}    + \frac{c}{c+k}\frac{c^p}{(c+k)^{2p-1}}\\
 &\geq & \frac{c^{p+1}}{c+k} \frac{(c+k-1)^{c-2p+1}-(c+1)^{c-2p+1}}{c-2p+1} + \frac{c}{c+k} \frac{c^p}{(c+k)^{2p-1}}\\
 &= & \frac{c^{p+1}}{c-2p+1}\frac{(c+k-1)^{c-2p+1}}{c+k} + O(1/k)\\ 
 &\overset{c\geq 2p+1}{\geq} & \frac{c^{p+1}}{c-2p+1} + O(1/k)\\
 &\overset{\text{$k$ big enough}}{\geq}& \frac{c^{p+1}}{c-1}
 \end{eqnarray*}

Therefore,
\begin{eqnarray*}
\bh_{k+1}    &\leq&  \frac{C_2(1-\frac{c^{p+1}}{c-1})}{(k+c)^p} + \frac{C_3}{(c+k)^{p+1}}.
\end{eqnarray*}
Define 
$\epsilon = \frac{2c^{p+1}}{c^{p+1} + 1 - c}$
and pick $C_2 > \frac{C_3}{c\epsilon}$. By assumption, $\epsilon > 0$. Consider $k > K$ such that for all $k$, 
\[
\frac{(k+c+1)^p}{(k+c)^p} \leq 1+\frac{\epsilon}{2}, \qquad \frac{C_3}{c+k} \leq \frac{\epsilon}{2}\frac{(c+k)^p}{(c+k+1)^p}. 
\]
Then
\begin{eqnarray*}
\bh_{k+1}    &\leq&  \frac{C_2(1-\tfrac{\epsilon}{2})}{(k+c+1)^p} + \frac{\tfrac{\epsilon}{2}}{(k+c+1)^p} \leq \frac{C_2}{(k+c+1)^p}.
\end{eqnarray*}
We have now proved the inductive step.
Picking $C_2 \geq \bh_0$ gives the starting condition, completing the proof.

\end{proof}

\section{Local rates}
\label{app:sec:local}
\begin{lemma}[Local convergence]
Define, for all  $t$,
\begin{equation}
\tilde s(t) = \argmin{\tilde s\in \conv(\mS(\bx^*))}\|s(t)-\tilde s\|_2, \qquad \hat s(t) = \bar \beta_t^{-1}\int_{0}^t \beta_{t,\tau} \tilde s(\tau) d\tau.
\label{eq:def-shat}
\end{equation}
e.g., $\hat s(t)$ is the closest point in the convex hull of the support of $\bx^*$ to the the point $s(t) = \lmo_\mD(x(t))$.

Then, 
\[
\|\bar s(t) - \hat s(t)\|_2 \leq \frac{c_1}{(c+t)^c}.
\]
\end{lemma}

The proof of this theorem actually does not really depend on how well the FW method works inherently, but rather is a consequence of the averaging. Intuitively, the proof states that after the manifold has been identified, all new support components must also be in the convex hull of the support set of the \emph{optimal} solution; thus, in fact $s_2(t) - \hat s(t) = 0$. However, because the accumulation term in the flow actually is not a true average until $t\to +\infty$, there is a pesky normalization term which must be accounted for. Note, importantly, that this normalization term does not appear in the method, where the accumulation weights always equal 1 (pure averaging).

\begin{proof}
First, note that 
\[
\bar s(t) - \hat s(t) =\bar s(t) - \bar \beta_t^{-1} \int_0^t \beta_{t,\tau} \tilde s(\tau) d\tau = \int_0^t \beta_{t,\tau} (s(\tau) - \bar\beta_t^{-1} \tilde s(\tau)) d\tau.
\]
Using triangle inequality,
\[
\|\bar s(t) - \hat s(t)\|_2 \leq \underbrace{ \|\int_0^{\bar t} \beta_{t,\tau} (s(\tau) - \bar\beta_t^{-1} \tilde s(\tau)) d\tau\|_2}_{\epsilon_1} + \underbrace{\|\int_{\bar t}^t \beta_{t,\tau} (s(\tau) - \bar\beta_t^{-1} \tilde s(\tau)) d\tau\|_2}_{\epsilon_2}.
\]

Expanding the first term, 
via Cauchy Scwhartz for integrals, we can write, elementwise,
\[
\int_0^{\bar t} \beta_{t,\tau} (s(\tau)_i - \bar\beta_t^{-1} \tilde s(\tau)_i) d\tau   \leq \int_0^{\bar t} \beta_{t,\tau}  d\tau \int_0^{\bar t} |s(\tau)_i - \bar\beta_{ t}^{-1} \tilde s(\tau)_i| d\tau
\]
and thus, 
\[
\|\int_0^{ t} \beta_{t,\tau} (s(\tau) - \bar\beta_t^{-1} \tilde s(\tau)) d\tau\|   \leq \int_0^{\bar t} \beta_{t,\tau}   d\tau \underbrace{\|\int_0^{ t} s(\tau)_i - \bar\beta_t^{-1} \tilde s(\tau)_i d\tau\|_2}_{\leq 2 D (1+\bar \beta_{t}^{-1}) \bar t} .
\]
and moreover,
\[
\int_0^{\bar t} \beta_{t,\tau} d\tau = \int_0^{\bar t} \frac{c(c+\tau)^{c-1}}{(c+t)^c} d\tau = \frac{\hat c_0}{(c+t)^c}
\]
since $\int_0^{\bar t} b(b+\tau)^{b-1}  d\tau$ does not depend on $t$.
Thus the first error term 
\[
\epsilon_1 \leq \frac{2\hat c_0 D \bar t (1+\bar \beta_t^{-1}) }{(c+t)^c} \leq \frac{c_0}{(c+t)^c}
\]
where
\[
\hat c_0 :=  2D\bar t \int_0^{\bar t} c(c+\tau)^{c-1} d\tau.
\]

In the second error term, \emph{because the manifold has now been identified}, $s(\tau)  = \tilde s(\tau)$, and so 
\[
\int_{\bar t}^t \beta_{t,\tau} (s(\tau) - \bar\beta_t^{-1} \tilde s(\tau)) d\tau   
=
\int_{\bar t}^t \beta_{t,\tau}(1-\bar \beta_t^{-1}) s(\tau)  d\tau   
\]
and using the same Cauchy-Schwartz argument, 
\[
\int_{\bar t}^t \beta_{t,\tau}(1-\bar \beta_t^{-1}) s(\tau)  d\tau \leq D\int_{\bar t}^t \beta_{t,\tau}(1-\bar \beta_t^{-1})    d\tau.
\]
The term
\[
1-\bar \beta_t^{-1} = |1-\frac{1}{1-(\frac{c}{c+t})^c}| =\frac{c^c}{(c+t)^c-c^c}\leq \frac{2c^2 }{(c+t)^{c}}
\] 
and thus
\[
\epsilon_t \leq \int_{\bar t}^t \beta_{t,\tau}(1-\bar\beta_t^{-1}) d\tau \leq \frac{2 c^3}{(c+t)^{2c}}\int_{\bar t}^t (c+\tau)^{c-1}  = 
\frac{2c^2 }{(c+t)^{2c}}( (c+t)^c - (c+\bar t)^c)  \leq \frac{2c^2}{(c+t)^c}.
\]
Thus,
\[
\|\bar s(t)-\hat s(t)\|_2  \leq \frac{\hat c_0+2c^2}{(c+t)^c} = O(\frac{1}{(c+t)^c}).
\]

\end{proof}

\begin{corollary}[Local flow rate]
Suppose that for all $x\in \mD$, $\|\nabla f(x)\|_2 \leq G$ for some $G$ large enough. 
Consider $\gamma(t) = \beta(t) = \frac{c}{c+t}$. Then the ODE
\[
\dot h(x(t)) = \gamma(t)\nabla f(x)^T(\bar s - x)
\]
has solutions 
  $h(t) = O(\frac{\log(t)}{(c+t)^c})$
   when $t \geq \bar t$.

\end{corollary}
\begin{proof}
First, we rewrite the ODE in a more familiar way, with an extra error term
\[
\dot h(x(t)) = \gamma(t)\nabla f(x)^T(\bar s - \hat s) +  \gamma(t)\nabla f(x)^T(\hat s - x)
\]
where $\hat s$ is as defined in \eqref{eq:def-shat}. By construction, $\hat s$ is a convex combination of $\tilde s\in \mS(\bx^*)$. Moreover, after $t \geq \bar t$, $\mS(\bar x(t)) = \mS(\bx^*)$, and thus 
\[
\nabla f(x)^T(\hat s(t) - x) = \nabla f(x)^T(s(t) - x) = -\gap(t) \leq -h(t).
\]
Then, using Cauchy-Schwartz, and piecing it together,
\[
h(t) = \nabla f(x)^T(\hat s(t) - x) \leq G\gamma(t)\|\bar s - \hat s\|_2 - \gamma(t) h(t) \leq \frac{G\gamma(t)  c_1}{(c+t)^c} - \gamma(t) h(t).
\]

 Let us therefore consider the system
  \[
  \dot h(x(t)) = \frac{2 G D \gamma(t) }{(c+t)^c}   - \gamma(t) h(x(t)) .
\]
The solution to this ODE is
 \[
 h(t) = \frac{h(0)c^c + 2GDc \log(c+t) - 2GDc\log(c)}{(c + t)^c} = O(\frac{\log(t)}{(c+t)^c}).
 \]

\end{proof}

\begin{lemma}[Local  averaging error]
\label{lem:local_avgerror}
Define, for all  $k$,
\[
\tilde \bs_k = \argmin{\tilde \bs\in \conv(\mS(\bx^*))}\|\bs_k-\tilde \bs\|_2, \qquad \hat \bs_k = \bar \beta_k^{-1}\sum_{i=1}^k \beta_{k,i} \tilde \bs_i.
\]
e.g., $\tilde \bs(k)$ is the closest point in the convex hull of the support of $\bx^*$ to the the point $\bs_k = \lmo_\mD(\bx_k)$.

Then, 
\[
\|\bar \bs_k - \hat \bs_k\|_2 \leq \frac{c_2}{k^c}.
\]
\end{lemma}

\begin{proof}
First, note that 
\[
\bar \bs_k - \hat \bs_k =\bar \bs_k -  \sum_{i=1}^k \beta_{k,i} \tilde \bs_i = \sum_{i=1}^k  \beta_{k,i} (\bs_i -  \tilde \bs_i).
\]
Using triangle inequality,
\[
\|\bar \bs_k - \hat \bs_k\|_2 \leq \underbrace{ \|\sum_{i=1}^{\bar k} \beta_{k,i} (\bs_i -  \tilde \bs_i)) \|_2}_{\epsilon} + \underbrace{\|\sum_{i=\bar k}^k \beta_{k,i} (\bs_i -  \tilde \bs_i) \|_2}_{0}.
\]
where the second error term is 0 since the manifold has been identified, so $\tilde \bs_i = \bs_i$ for all $i \geq \bar k$.

Expanding the first term,  using a Holder norm (1 and $\infty$ norm) argument,
\begin{eqnarray*}
\|\sum_{i=1}^{\bar k} \beta_{k,i} (\bs_i -  \tilde \bs_i)  \|_2     &\leq& 2D\sum_{i=1}^{\bar k} \beta_{k,i} \\
&=&  2D\sum_{i=1}^{\bar k} \frac{c}{(c+i)}\prod_{j=0}^{ k - i - 1}(1-\frac{c}{(c+k-j)}) \\
&=& 2D\sum_{i=1}^{\bar k} \frac{c}{(c+i)} \prod_{j=0}^c \frac{i-1+j}{c+k-j}\\
&\leq& 2D (\frac{\bar k - 1 + c}{k})^c\sum_{i=1}^{\bar k} \frac{c}{(c+i)}= O(1/k^c).
\end{eqnarray*}
\end{proof}

\begin{corollary}[Local convergence rate bounds]
\label{lem:local_convratebnd}
Suppose that for all $\bx\in \mD$, $\|\nabla f(\bx)\|_2 \leq G$ for some $G$ large enough. Define also $r$ the decay constant of $\|\bar\bs_k-\bx_k\|^2_2$ ($=O(1/k^r)$).   
Consider $\gamma_k = \beta_k = \frac{c}{c+k}$. Then the difference equation
\[
\bh(\bx_{k+1})-\bh(\bx_k) \leq \gamma_k\nabla f(\bx)^T(\bar \bs_k - \bx_k) + \frac{C}{k^r}
\]
is satisfied with candidate solution 
 $\bh(\bx_k) = C_4 \max\{\frac{\log(k)}{(c+k)^c},\frac{1}{k^{r+1}}\}$
   when $k \geq \bar k$.

\end{corollary}

\begin{proof}
First, we rewrite the ODE in a more familiar way, with an extra error term
\[
\bh(\bx_{k+1})-\bh(\bx_k) = \gamma_k\underbrace{\nabla f(\bx_k)^T(\bar \bs_k - \hat \bs_k)}_{\leq \gamma_k G\|\bar\bs_k-\hat\bs_k\|_2} +  \gamma_k\nabla f(\bx_k)^T(\hat \bs_k - \bx_k) + \frac{C}{k^{r+2}}
\]
where $\hat \bs_k$ is as defined in \eqref{eq:def-shat}. By construction, $\hat \bs_k$ is a convex combination of $\tilde \bs_i\in \mS(\bx^*)$. Moreover, after $k \geq \bar k$, $\mS(\bar \bx_k) = \mS(\bx^*)$, and thus 
\[
\nabla f(\bx_k)^T(\hat \bs_k - \bx_k) = \nabla f(\bx_k)^T(\bs_k - \bx_k) = -\gap(\bx_k) \leq -\bh(\bx_k).
\]
Then,   piecing it together,
\[
\bh(\bx_{k+1})-\bh(\bx_k) \leq \gamma_kG\|\bar\bs-\hat\bs\|_2 -\gamma_k\bh(\bx_k) + \frac{C}{k^{r+2}} \overset{\text{Lemma \ref{lem:local_avgerror}}}{\leq } \underbrace{\gamma_k\frac{GC_2}{k^c}}_{\leq C_3/k^{c+1}} -\gamma_k\bh(\bx_k) + \frac{C}{k^{r+2}}
\]

Recursively, we can now show that for $C_4 \geq C+C_3$, if
\[
\bh(\bx_k) \leq C_4 \max\{\frac{\log(k)}{(c+k)^c},\frac{1}{k^{r+1}}\}
\]
then, 
\begin{eqnarray*}
\bh(\bx_{k+1})&\leq& \frac{C_3}{k^{c+1}}+ \frac{C}{k^{r+1}}+(1-\gamma_k) \bh(\bx_k) \\
&\leq & \frac{C_3}{k^{c+1}}+ \frac{C}{k^{r+1}}+\frac{k}{c+k}  C_4\max\{\frac{\log(k)}{(c+k)^c},\frac{1}{k^{r+1}}\}.
\end{eqnarray*}
If $c\leq r+1$ then 

\begin{eqnarray*}
\bh(\bx_{k+1})&\leq&  \frac{C+C_3}{k^c}+\frac{k}{c+k}  \frac{C_4\log(k)}{(c+k)^c}\leq \frac{k}{c+k}  \frac{C_4\log(k)}{(c+k)^c}\leq   \frac{C_4\log(k+1)}{(c+k+1)^c}
\end{eqnarray*}
for $k$ large enough. Otherwise,

\begin{eqnarray*}
\bh(\bx_{k+1})&\leq&  \frac{C+C_3}{k^{r+2}}+\frac{k}{c+k}  \frac{C_4}{k^{r+1}}\leq \frac{C_4}{(k+1)^{r+1}}.
\end{eqnarray*}
for $k$ large enough.
\end{proof}

\begin{theorem}[Local convergence rate]
Picking $c \geq 3p/2+1$, the proposed method \AvgFW~has an overall convergence $\bh(\bx_k) = O(k^{-3p/2})$.
\end{theorem}

\begin{proof}
Putting together Theorem \ref{th:globalrate}, Lemma \ref{lem:local_avgerror}, and Corollary \ref{lem:local_convratebnd}, we can resolve the constants
\[
q = p, \qquad r = \min\{q/2+p-1,2p\} = \frac{3p}{2}-1
\]
and  resolves an overall convergence bound of 
$\bh(\bx_k) = O(k^{-3p/2})$.

\end{proof}

\end{document}